\documentclass[12pt]{article}
\usepackage{amsmath,amsthm,amssymb,enumerate,eucal,a4,graphicx,subfig}

\newtheorem{theorem}{Theorem}
\newtheorem{lemma}[theorem]{Lemma}
\newtheorem{proposition}[theorem]{Proposition}

\newtheorem{conjecture}[theorem]{Conjecture}

\newtheorem{problem}[theorem]{Problem}
\newtheorem{observation}[theorem]{Observation}

\newcommand\size[1] {\left|{#1}\right|}

\newcommand\Setx[1] {\left\{{#1}\right\}}
\newcommand\sub {\subseteq}

\DeclareMathOperator{\Ass}{Ass}
\DeclareMathOperator{\depth}{depth}

\newcommand{\s}{\mathord}
\newcommand{\sseqx}[1]{\hspace{0.8pt}#1}
\newcommand{\sseq}[1]{(\sseqx{#1})}

\newcommand{\repl}[2]{#1^#2}
\newcommand{\cart}{\Box}
\newcommand{\upto}[1]{[#1]}
\newcommand{\pattsp}{\hspace*{0.5pt}}
\newcommand{\patt}[3]{\mbox{$#1$\pattsp$\cdot$\pattsp $#2$\pattsp$\cdot$\pattsp $#3$}}

\newcommand{\ZZ}{\mathbb Z}

\newtheorem{xcasehdr}{Case}
\newtheorem{xsubcasehdr}{Subcase}[xcasehdr]

\newenvironment{xcase}[1]%
{\vspace{-1mm}\par\noindent\xcasehdr{#1}\upshape\vspace{2mm}\par\noindent}%
{\vspace{2mm}}

{\vspace{-1mm}\par\noindent \xsubcasehdr{#1}\upshape
  \vspace{2mm}\par\noindent}%
{\hspace*{0mm}\par\vspace{2mm}}

\newcommand{\hf}{\hspace*{0pt}\hspace{\fill}\hspace*{0pt}}
\newcommand{\fig}[1]{\includegraphics[page=#1]{repl-figs.pdf}}%
\newcommand{\sfig}[2]{\subfloat[#2]{\fig{#1}}}%

\title{\textbf{Replication in critical graphs\\and the persistence of
    monomial ideals\thanks{Research partially supported by the
      TEOMATRO grant ANR-10-BLAN 0207 ``New Trends in Matroids: Base
      Polytopes, Structure, Algorithms and Interactions''.}}}

\author{Tom\'{a}\v{s} Kaiser$^1$\and Mat\v{e}j Stehl\'{\i}k$^2$ \and Riste
  \v{S}krekovski$^3$}

\date{}

\begin{document}
\maketitle
\footnotetext[1]{Department of Mathematics, Institute for Theoretical
  Computer Science (CE-ITI) and European Centre of Excellence
  NTIS---New Technologies for Information Society, University of West
  Bohemia, Univerzitn\'{\i}~8, 306~14~Plze\v{n}, Czech
  Republic. E-mail: \texttt{kaisert@kma.zcu.cz}. Supported by project
  P202/12/G061 of the Czech Science Foundation.}%
\footnotetext[2]{UJF-Grenoble 1 / CNRS / Grenoble-INP, G-SCOP UMR5272
  Grenoble, F-38031, France. E-mail:
  \texttt{matej.stehlik@g-scop.inpg.fr}.}%
\footnotetext[3]{Department of Mathematics, University of Ljubljana,
  Ljubljana \& Faculty of Information Studies, Novo Mesto \& FAMNIT,
  University of Primorska, Koper, Slovenia. E-mail:
  \texttt{skrekovski@gmail.com}. Partially supported by ARRS Program
  P1-0383 and by the French-Slovenian bilateral project
  BI-FR/12-13-Proteus-011.}%


\begin{abstract}
  Motivated by questions about square-free monomial ideals in
  polynomial rings, in 2010 Francisco et al.\ conjectured that for
  every positive integer $k$ and every $k$-critical (i.e., critically
  $k$-chromatic) graph, there is a set of vertices whose replication
  produces a $(k+1)$-critical graph. (The replication of a set $W$ of
  vertices of a graph is the operation that adds a copy of each vertex
  $w$ in $W$, one at a time, and connects it to $w$ and all its
  neighbours.)

  We disprove the conjecture by providing an infinite family of
  counterexamples. Furthermore, the smallest member of the family
  answers a question of Herzog and Hibi concerning the depth functions
  of square-free monomial ideals in polynomial rings, and a related
  question on the persistence property of such ideals.
\end{abstract}


\section{Introduction}
\label{sec:intro}

An investigation of the properties of square-free monomial ideals in
polynomial rings led Francisco et al.~\cite{FHT-conjecture} to an
interesting question about replication in colour-critical graphs that
we answer in the present paper.

In the area of graph colourings, constructions and properties of
colour-critical graphs are a classical subject (see, e.g.,
\cite[Section 14.2]{BM-graph}). The replication of a set of vertices,
whose definition we will recall shortly, is a natural operation in
this context. It is also of central importance for the theory of
perfect graphs (cf.~\cite[Chapter 65]{Sch-combinatorial}).

For the terminology and notation of graph theory, we follow Bondy and
Murty~\cite{BM-graph}. We deal with graphs without parallel edges and
loops. The vertex set and the edge set of a graph $G$ are denoted by
$V(G)$ and $E(G)$, respectively.

A graph $G$ is \emph{$k$-chromatic} if its chromatic number is $k$. It
is \emph{$k$-critical} if $G$ is $k$-chromatic and $G-v$ is
$(k-1)$-colourable for each vertex $v$ of $G$. Furthermore, $G$ is
\emph{$k$-edge-critical} if $G$ is $k$-chromatic and every proper
subgraph of $G$ is $(k-1)$-colourable.

\emph{Replicating} (also \emph{duplicating}) a vertex $w\in V(G)$
means adding a copy (or \emph{clone}) $w'$ of $w$ and making it
adjacent to $w$ and all its neighbours. To replicate a set $W\sub
V(G)$, we replicate each vertex $w\in W$ in sequence. The resulting
graph $\repl G W$ is independent of the order in which the individual
vertices are replicated.

Francisco et al.~\cite{FHT-conjecture} posed the following conjecture:
\begin{conjecture}\label{conj:main}
  For any positive integer $k$ and any $k$-critical graph $G$, there
  is a set $W\sub V(G)$ such that $\repl G W$ is $(k+1)$-critical.
\end{conjecture}
In Section~\ref{sec:counterexample} of the present paper, we disprove
the conjecture by showing that each member of an infinite family of
4-critical graphs constructed by Gallai~\cite{Gal-kritische} is a
counterexample. In Section~\ref{sec:algebra}, we discuss the algebraic
properties of the smallest member of this family and show that it also
answers two open questions concerning square-free monomial ideals in
polynomial rings. Thus, the result provides a nice example of
interplay and useful exchange between algebra and combinatorics.


\section{A counterexample}
\label{sec:counterexample}

Gallai's construction~\cite{Gal-kritische} of an infinite family of
4-regular 4-edge-critical graphs provided the first example of a
$k$-edge-critical graph without vertices of degree $k-1$. The
definition can be expressed as follows.

For a positive integer $n$, let $\upto n$ denote the set
$\Setx{0,\dots,n-1}$. Let $P_n$ be a path with vertex set $\upto
n$, with vertices in the increasing order along $P_n$. Let $K_3$
be the complete graph whose vertex set is the group $\ZZ_3$.

For $n\geq 4$, we define $H_n$ as the graph obtained from the
Cartesian product $P_n \cart K_3$ by adding the three edges
joining $(0,j)$ to $(n-1,-j)$ for $j\in\ZZ_3$. (See
Figure~\ref{fig:graph}a.)

The 4-regular graphs $H_n$ are interesting in various ways; for
instance, they embed in the Klein bottle as quadrangulations
(cf. Figure~\ref{fig:graph}b). In this section, we show that Gallai's
graphs are counterexamples to Conjecture~\ref{conj:main}:
\begin{theorem}\label{t:main}
  For any $n\geq 4$ and any $W\sub V(H_n)$, the graph $\repl{H_n}W$ is
  not $5$-critical.
\end{theorem}
It is interesting to note that by~\cite[Theorem~1.3]{FHT-conjecture},
Conjecture~\ref{conj:main} holds for graphs $G$ satisfying $\chi_f(G)
> \chi(G)-1$, where $\chi$ denotes the chromatic number and $\chi_f$
denotes the fractional chromatic number (see, e.g.,
\cite[Definition~3.8]{FHT-conjecture} for the definition). Since the
graphs $H_n$ are $4$-chromatic and their fractional chromatic number
equals $3$, they show that the bound in Theorem~1.3 of
\cite{FHT-conjecture} cannot be improved.

We will divide the proof of Theorem~\ref{t:main} into two
parts. First, we show that for certain sets $W$, the chromatic number
of $\repl{H_n}W$ is at least 5, but $\repl{H_n}W$ is not 5-critical
(Lemma~\ref{l:five}). We then prove that for any other set $W$,
$\repl{H_n}W$ is 4-chromatic (Proposition~\ref{p:four}).

Let $i\in\upto n$ and $j\in\ZZ_3$. The \emph{$i$-th column} of $H_n$
is the set $C_i = \Setx i \times \ZZ_3$. Similarly, the \emph{$j$-th
  row} of $H_n$ is $R_j = \upto n \times \Setx j$. The vertex in
$C_i\cap R_j$ is denoted by $v_{i,j}$. In accordance with the notation
introduced above, the clone of $v_{i,j}\in W$ in $\repl{H_n}W$ is
denoted by $v'_{i,j}$.

We introduce notation for certain subgraphs of $\repl{H_n}W$. Let
$i\in\upto n$. We define $X_i$ as the clique in $\repl{H_n}W$ on
the vertices in $C_i$ and their clones. Furthermore, $Y_i$ is the
induced subgraph of $\repl{H_n}W$ on $V(X_i)\cup V(X_{i+1})$ (addition
modulo $n$).

\begin{figure}
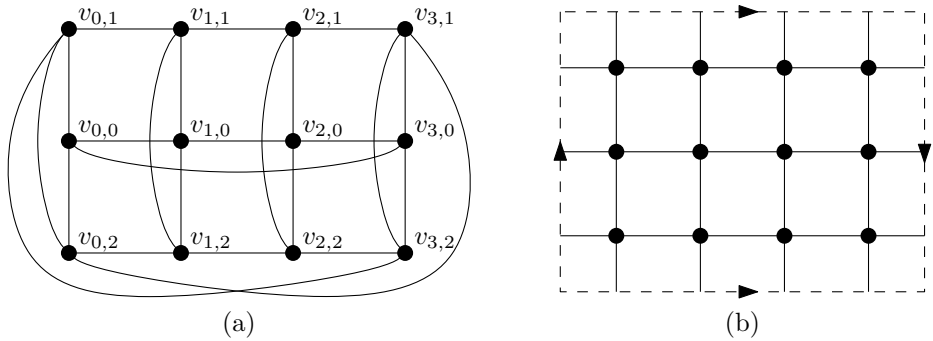

  \centering
  \hf\sfig{2}{}\hf\sfig{3}{}\hf
  \caption{(a) The graph $H_4$. (b) A drawing of $H_4$ as a
    quadrangulation of the Klein bottle. The opposite sides of the
    bounding rectangle are identified in such a way that the arrows
    match.}
  \label{fig:graph}
\end{figure}

\begin{lemma}\label{l:five}
  Let $n\geq 4$ and let $W\sub V(H_n)$. In each of the following
  cases, the graph $\repl{H_n}W$ has chromatic number at least $5$ and
  is not $5$-critical:
  \begin{enumerate}[\quad$(a)$]
  \item there is some $i\in\upto n$ such that the set $W\cap C_i$ has
    size at least 2,
  \item $W$ contains at least $n-1$ vertices of $R_0$ and $n$ is odd,
  \item the induced subgraph of $H_n$ on $W-R_0$ contains a path with
    at least $n$ vertices and $n$ is even.
  \end{enumerate}
\end{lemma}
\begin{proof}
  (a) Suppose that $W \cap C_i$ has size at least 2, so $\size{V(X_i)}
  \geq 5$. Since $\repl{H_n}W$ contains the clique $X_i$ as a proper
  subgraph, it is neither 4-colourable nor 5-critical.

  (b) Without loss of generality, assume that $W$ contains
  $R_0-\Setx{v_{n-1,0}}$. Furthermore, suppose that $n$ is odd. For
  contradiction, let $c$ be a 4-colouring of $\repl{H_n}W$. By
  symmetry, the vertices $v_{0,0}$ and $v'_{0,0}$ may be assumed to
  have colours 1 and 2 in $c$. This forces the pairs of colours
  assigned to $v_{i,0}$ and $v'_{i,0}$ alternate between $\Setx{1,2}$
  and $\Setx{3,4}$ as $i$ increases. Hence, $v_{n-1,0}$ has neighbours
  of all four colours, a contradiction which shows that $\repl{H_n}W$
  is not 4-colourable. Because the argument involves only vertices in
  $R_0$ and their clones, it implies that, say, $\repl{H_n}W-v_{0,2}$
  is not 4-colourable. It follows that $\repl{H_n}W$ is not
  5-critical.

  (c) Suppose that $n$ is even and the induced subgraph of $W-R_0$
  contains a path with at least $n$ vertices. By symmetry, we may
  assume that $R_1 \sub W$. We prove that $\repl{H_n}W$ is not
  4-colourable. Suppose the contrary and consider a 4-colouring of
  $\repl{H_n}W$. An argument similar to the one used in part (b)
  implies that the vertices $v_{0,1}$, $v'_{0,1}$, $v_{n-1,1}$ and
  $v'_{n-1,1}$ have distinct colours. Since they have a common
  neighbour $v_{n-1,2}$, we obtain a contradiction. In the same manner
  as above, it follows that $\repl{H_n}W$ is not 5-critical.
\end{proof}

\begin{lemma}\label{l:max}
  If $W\sub H_n$ satisfies none of the conditions $(a)$--$(c)$ in
  Lemma~\ref{l:five}, then there is a set $Z$ such that $W\sub Z\sub
  V(H_n)$, $Z$ contains exactly one vertex from each $C_i$
  ($i\in\upto n$) and $Z$ still satisfies none of $(a)$--$(c)$.
\end{lemma}
\begin{proof}
  Since $W$ does not satisfy condition (a), it contains at most one
  vertex from each set $C_i$ ($i\in\upto n$). Suppose that $W\cap C_i
  = \emptyset$ for some $i$. We claim that conditions (a)--(c) are
  still violated for the set $W\cup\Setx{w}$, for some $w\in C_i$. If
  $W\cup\Setx{v_{i,0}}$ satisfies any of the conditions, it must be
  condition (b), which means that $n$ is odd. In that case,
  $W\cup\Setx{v_{i,1}}$ trivially fails to satisfy the conditions. By
  adding further vertices in this way, we arrive at a set $Z$ with the
  desired properties.
\end{proof}

Before we embark on the proof of Proposition~\ref{p:four}, it will be
convenient to introduce some terminology. Assume that $W\sub V(H_n)$
is a set which satisfies none of the conditions in
Lemma~\ref{l:five}. In addition, we will assume that 
\begin{equation}\label{eq:total}
  \text{$W$ intersects
    each $C_i$ ($i\in\upto n$) in exactly one vertex.}
\end{equation}
For each $i\in\upto n$, we will define $w_i$ to be the unique element
of $\ZZ_3$ such that $W\cap C_i = \Setx{v_{i,w_i}}$. (In the proof of
Proposition~\ref{p:four} below, we will ensure
condition~\eqref{eq:total} by appealing to Lemma~\ref{l:max}.)

We will encode the set $W$ into a sequence of signs, defined as
follows. A \emph{sign sequence} $\sigma$ is a sequence of elements of
$\ZZ_3$. We will often write `$+$' for the element 1 and `$-$' for the
element 2 (which coincides with $-1$). Thus, the sign sequence
$\sseq{\s0 \s+ \s- \s+}$ stands for the sequence $(0,1,2,1)$.

To the set $W$, we assign the sign sequence $\sigma^W = s_0\dots
s_{n-1}$, where each $s_i\in\ZZ_3$ is defined as
\begin{align*}
  s_i =
  \begin{cases}
    w_{i+1} - w_i & \text{if $0 \leq i \leq n-2$,}\\
    -w_0 - w_{n-1} & \text{if $i=n-1$.}
  \end{cases}
\end{align*}
The change of sign in the latter case reflects the fact that the
vertex $v_{n-1,j}$ is adjacent to $v_{0,-j}$ rather than $v_{0,j}$. It
may be helpful to view $H_n$ as the graph obtained from the Cartesian
product $P_{n+1}\cart K_3$ by identifying the vertex $(0,j)$ with
$(n,-j)$ for each $j\in\ZZ_3$. It is then natural to define $w_n =
-w_0$, in which case $s_{n-1}$ is precisely $w_n-w_{n-1}$.

To describe a 4-colouring of the clique $X_i$ in $\repl{H_n}W$
($i\in\upto n$), we introduce the notion of a \emph{pattern}. This is
a cyclically ordered partition of the set $\Setx{1,2,3,4}$ into three
\emph{parts}, with one part of size 2 and the remaining parts of size
1. The two colours contained in the part of size 2 are
\emph{paired}. Two patterns differing only by a cyclic shift of the
parts are regarded as identical. Given a 4-colouring $c$ of $X_i$, the
corresponding \emph{pattern at $X_i$} is
\begin{equation*}
  \pi_i(c) = \Bigl(\Setx{c(v_{i,w_i}),c(v'_{i,w_i})},\Setx{c(v_{i,w_i+1})},\Setx{c(v_{i,w_i+2})}\Bigr).
\end{equation*}
We use a more concise notation for patterns: for instance, instead of
writing $(\Setx{1,2},\Setx 3, \Setx 4)$ we write just
$\patt{12}34$. Note that a pattern does not determine the colouring
uniquely since it does not specify the order of the paired colours.

We now determine the possible combinations of patterns at $X_i$ and at
$X_{i+1}$ in a valid colouring of $Y_i$. Suppose that $c_0$ is a
colouring of $X_0$ with pattern $\patt{12}34$, and let $s =
w_1-w_0$. Consider first the case that $s=1$. It is routine to check
that for any valid extension of $c_0$ to $Y_0$, the pattern at $X_1$
is $\patt{12}34$, $\patt{14}23$ or $\patt{24}13$
(cf.~Figure~\ref{fig:patt}). Conversely, each of these patterns
determines a valid extension.

\begin{figure}
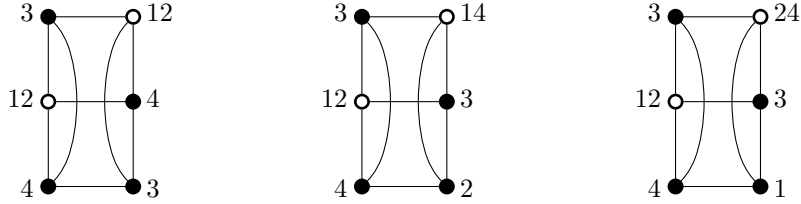

  \centering
  \hf\fig6\hf\fig4\hf\fig5\hf
  \caption{Valid colourings of $Y_0$ such that the pattern at $X_0$ is
    $\patt{12}34$. The colouring of $Y_0$ is represented in the
    induced subgraph of $H_n$ on $C_0\cup C_1$ by assigning a pair of
    colours to each vertex in $W$. These vertices are shown as
    circles, the other vertices as solid dots.}
  \label{fig:patt}
\end{figure}

Considering the other possibilities for $s$, we find that the sets of
patterns at $X_1$ corresponding to valid extensions of $c_0$ are as
follows:
\begin{center}
  \begin{tabular}{cccl}
    $\patt{12}34$ & $\patt{14}23$ & $\patt{24}13$ & \qquad if $s=1$,\\
    $\patt{12}34$ & $\patt{13}42$ & $\patt{23}41$ & \qquad if $s=-1$,\\
    $\patt{34}12$ & $\patt{34}21$ & & \qquad if $s=0$.
  \end{tabular}
\end{center}
The patterns in the first row of the above table are said to be
\emph{$+$-compatible} with $\patt{12}34$. The notions of
\emph{$-$-compatibility} and \emph{$0$-compatibility} are defined in
a similar way using the second and third row, respectively. Applying
a suitable permutation to the set of colours, we can extend these
definitions to any other pattern in place of $\patt{12}34$.

The same discussion applies just as well to patterns at $X_i$ and
$X_{i+1}$, where $1\leq i \leq n-2$. For $i=n-1$, we need to take into
account the `twist' in $Y_{n-1}$. We find that for a valid colouring
of $Y_{n-1}$, the pattern $\pi$ induced at $X_{n-1}$ and the pattern
$\rho$ induced at $X_0$ have the property that $\overline\rho$ is
$s'$-compatible with $\pi$, where $\overline\rho$ is the
\emph{reverse} of $\rho$, i.e., the pattern obtained by reversing the
order of parts in $\rho$, and $s' = -w_0-w_{n-1}$.

There is a simple description of the patterns that are $+$-compatible
with a given pattern $\pi=\patt{xy}zw$. One of them is $\pi$
itself. To obtain the other ones, choose a colour that is paired in
$\pi$ ($x$ or $y$) and move it to the preceding part of $\pi$ with
respect to the cyclic ordering. Reversing the direction of the move,
we obtain the $-$-compatible patterns. Finally, to obtain the two
$0$-compatible patterns, merge the two colours that are unpaired in
$\pi$ into one part, and put the other two colours into two parts,
choosing any of the two possible orderings.

We represent the notion of compatibility of patterns using an
auxiliary graph $D$, in which we allow both directed and undirected
edges as well as directed loops. The vertex set of $D$ is the set of
all 12 patterns. Patterns $\pi$ and $\rho$ are joined by an
undirected edge if they are $0$-compatible. There is a directed edge
from $\pi$ to $\rho$ if $\rho$ is $+$-compatible with $\pi$ (or
equivalently, if $\pi$ is $-$-compatible with $\rho$). In
particular, $D$ has a directed loop on each vertex. The graph $D$ is
shown in Figure~\ref{fig:aux} (with the loops omitted).

\begin{figure}
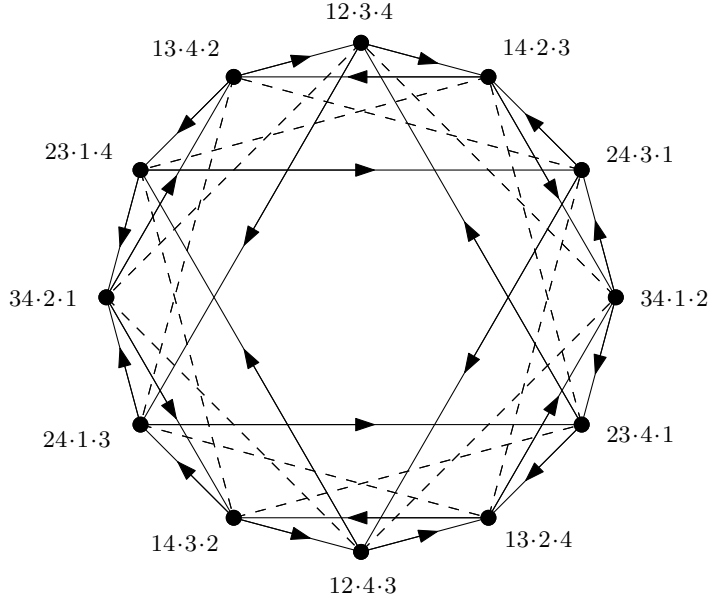

  \centering
  \fig{1}
  \caption{The auxiliary graph $D$. Directed loops at all the
    vertices are not shown.}
  \label{fig:aux}
\end{figure}

Let $\sigma = s_0\dots s_k$ be a sign sequence. A
\emph{$\sigma$-stroll} $S$ is a sequence $\pi_0\pi_1\dots \pi_{k+1}$,
where each $\pi_i$ ($0 \leq i \leq k+1$) is a vertex of $D$ and one of
the following conditions holds for each $j$ ($0 \leq j \leq k$):
\begin{itemize}
\item $s_j = 0$ and $D$ contains an undirected edge with endvertices
  $\pi_j$ and $\pi_{j+1}$,
\item $s_j = 1$ and there is a directed edge from $\pi_j$ to
  $\pi_{j+1}$,
\item $s_j = -1$ and there is a directed edge from $\pi_{j+1}$ to
  $\pi_j$.
\end{itemize}
For $s_j = \pm 1$, the directed edge is allowed to be a loop,
reflecting the fact that a pattern is both $+$-compatible and
$-$-compatible with itself.  A $\sigma$-stroll as above is said to
start at $\pi_0$ and end at $\pi_{k+1}$ (or to be a $\sigma$-stroll
from $\pi_0$ to $\pi_{k+1}$).

To illustrate the definition, if $\sigma=(\s- \s- \s+ \s0 \s+ \s-)$,
then a $\sigma$-stroll from $\patt{12}34$ to $\patt{12}43$ is
\begin{equation*}
  (\patt{12}34, \patt{13}42, \patt{34}21, \patt{14}32, \patt{23}41,
  \patt{13}24, \patt{12}43).
\end{equation*}

A sign sequence $\sigma$ is said to be \emph{reversing} if there is a
$\sigma$-stroll from $\patt{12}34$ to $\patt{34}12$. Note that by
interchanging colours 1 and 2, one can then obtain a $\sigma$-stroll
from $\patt{12}34$ to $\patt{34}21$ as well. Furthermore, $\sigma$ is
\emph{good} if there exists a $\sigma$-stroll from $\patt{12}34$ to
$\patt{12}43$. The latter terminology is justified by the following
lemma.

\begin{lemma}\label{l:good}
  If $\sigma^W$ is good, then the graph $\repl{H_n}W$ is $4$-colourable.
\end{lemma}
\begin{proof}
  Let $\sigma^W = s_0\dots s_{n-1}$ and let $S=\pi_0\dots \pi_n$ be a
  $\sigma^W$-stroll from $\patt{12}34$ to $\patt{12}43$. For each
  $i=0,\dots,n-1$, colour the vertices of $X_i$ in such a way that the
  pattern is $\pi_i$. By the definition, each $\pi_i$ ($0 \leq i\leq
  n-2$) is $s_i$-compatible with $\pi_{i+1}$, and so $Y_i$ is properly
  coloured.

  It remains to check the colouring of $Y_{n-1}$. As observed above,
  $Y_{n-1}$ is properly coloured if the reverse of $\pi_0$ (that is,
  $\patt{12}43$) is $s_{n-1}$-compatible with $\pi_{n-1}$. This is
  ensured by the requirement that $S$ ends at $\patt{12}43$.
\end{proof}

For a sign sequence $\sigma$, we define $-\sigma$ to be the sign
sequence obtained by replacing each $-$ sign by $+$ and vice versa.

\begin{lemma}\label{l:symm}
  If $\sigma$ is good, then $-\sigma$ is good.
\end{lemma}
\begin{proof}
  By inspecting Figure~\ref{fig:aux} or directly from the definition,
  one can see that if $D$ contains a directed edge from $\pi$ to
  $\rho$, then it also contains a directed edge from $\overline\rho$
  to $\overline\pi$, and a similar claim holds for the undirected
  edges. It follows that if $S=(\pi_0,\dots,\pi_k)$ is a
  $\sigma$-stroll, then $\overline S =
  (\overline{\pi_0},\dots,\overline{\pi_k})$ is a
  $(-\sigma)$-stroll. If $S$ is good, then $\overline S$ starts at
  $\patt{12}43$ and ends at $\patt{12}34$. Interchanging colours $3$
  and $4$ in each pattern in $\overline S$, we obtain a
  $(-\sigma)$-stroll from $\patt{12}34$ to $\patt{12}43$.
\end{proof}

Let $\sigma=s_0\dots s_{k-1}$ be a sign sequence and let $\pi$ and
$\rho$ be patterns such that $\pi\rho$ is an undirected edge of
$D$. We define a $\sigma$-stroll $S(\sigma;\pi,\rho) =
\pi_0\dots\pi_k$ by the following rule:
\begin{itemize}
\item $\pi_0 = \pi$,
\item if $s_i \neq 0$, then $\pi_{i+1}=\pi_i$ (where $0 \leq i \leq
  k-1$),
\item if $s_i = 0$, then $\pi_{i+1}$ is the vertex in
  $\Setx{\pi,\rho}$ distinct from $\pi_i$ (where $0 \leq i \leq k-1$).
\end{itemize}

Let $\sigma_1$ and $\sigma_2$ be sign sequences and let $\sigma$ be
their concatenation. If $P = (\pi_0,\dots,\pi_k)$ is a
$\sigma_1$-stroll and $R = (\rho_0,\dots,\rho_\ell)$ is a
$\sigma_2$-stroll such that $\pi_k = \rho_0$, then the
\emph{composition} of $P$ and $R$ is the $\sigma$-stroll
\begin{equation*}
  P\circ R = (\pi_0,\dots,\pi_k,\rho_1,\dots,\rho_\ell).
\end{equation*}

For any sign sequence $\sigma$, we let $z(\sigma)$ denote the number
of occurrences of the symbol $0$ in $\sigma$, reduced modulo 2. For
clarity, we omit one pair of parentheses in expressions such as
$z(\sseq{\s0 \s+ \s-})$.

\begin{observation}\label{obs:stationary}
  Let $\sigma$ be a sign sequence and $\pi,\rho$ be patterns such that
  $\pi\rho$ is an undirected edge of $D$. Then the $\sigma$-stroll
  $S(\sigma;\pi,\rho)$ starting at $\pi$ satisfies the following:
  \begin{enumerate}[\quad$(i)$]
  \item if $z(\sigma)=0$, then $S(\sigma;\pi,\rho)$ ends at $\pi$,
  \item otherwise, $S(\sigma;\pi,\rho)$ ends at $\rho$.
  \end{enumerate}
\end{observation}

We define an order $\preceq$ on sign sequences. Let $\tau,\sigma$ be
two sign sequences, where $\sigma=s_0\dots s_k$. We define $\tau
\preceq\sigma$ if there are indices $0\leq i_0<i_1\dots<i_m\leq k$
such that:
\begin{itemize}
\item $\tau = s_{i_0}s_{i_1}\dots s_{i_m}$,
\item $z(s_0s_1\dots s_{i_0-1}) = 0$, and
\item for every $j$ such that $0\leq j \leq m-1$, $z(s_{i_j+1}\dots
  s_{i_{j+1}-1}) = 0$.
\end{itemize}

\begin{lemma}\label{l:sub}
  Let $\sigma$ and $\tau$ be sign sequences such that
  $\tau\preceq\sigma$. The following holds:
  \begin{enumerate}[\quad$(i)$]
  \item if $z(\sigma) = z(\tau)$ and $\tau$ is good, then $\sigma$ is good,
  \item if $z(\sigma) \neq z(\tau)$ and $\tau$ is reversing, then
    $\sigma$ is good.
  \end{enumerate}
\end{lemma}
\begin{proof}
  (i) Suppose that $\sigma=s_0\dots s_k$ and $\tau = s_{i_0}\dots
  s_{i_m}$, where $0 \leq i_0 < \dots < i_m \leq k$. Let
  $S=(\rho_0,\dots,\rho_m)$ be a $\tau$-stroll from $\patt{12}34$ to
  $\patt{12}43$. For simplicity, set $i_{-1}=-1$ and for $j=0,\dots,m$
  let $\sigma^j$ be the subsequence of $\sigma$ from $s_{i_{j-1}+1}$
  to $s_{i_j-1}$. For each $j$, $0 \leq j \leq m$, choose a pattern
  $\epsilon_j$ such that $\rho_j\epsilon_j$ is an undirected edge of
  $D$.

  By the definition of the order $\preceq$, we have $z(\sigma^j) = 0$
  for each $j$, $0\leq j \leq m$. Observation~\ref{obs:stationary}(i)
  implies that $S(\sigma^j;\rho_j,\epsilon_j)$ is a $\sigma^j$-stroll from
  $\rho_j$ to $\rho_j$. Thus, the composition
  \begin{align*}
    S' = S(\sigma^0;\rho_0,\epsilon_0)&\circ(\rho_0,\rho_1)\circ
    S(\sigma^1;\rho_1,\epsilon_1) \circ (\rho_1,\rho_2) \circ\cdots\\
    &\circ
    S(\sigma^{m-1};\rho_{m-1},\epsilon_{m-1}) \circ (\rho_{m-1},\rho_m)
  \end{align*}
  is a valid $\sigma$-stroll from $\patt{12}34$ to $\patt{12}43$.

  Let $\sigma^{m+1}$ denote the sequence $s_{i_m+1}\dots s_k$. Then
  \begin{equation}\label{eq:tail}
    z(\sigma^{m+1}) \equiv z(\sigma) - z(\tau) \pmod 2.
  \end{equation}
  By \eqref{eq:tail}, $z(\sigma^{m+1}) = 0$ and so
  $S'':=S(\sigma^{m+1};\patt{12}43,\patt{12}34)$ is a
  $\sigma^{m+1}$-stroll from $\patt{12}43$ to $\patt{12}43$ by
  Observation~\ref{obs:stationary}(i). The $\sigma$-stroll $S' \circ
  S''$ then shows that $\sigma$ is good.

  The proof of (ii) is similar, except that $S$ is now a $\tau$-stroll
  from $\patt{12}34$ to $\patt{34}12$. Furthermore, $z(\sigma^{m+1}) =
  1$ and $S'' := S(\sigma^{m+1};\patt{34}12,\patt{12}43)$ is a
  $\sigma^{m+1}$-stroll from $\patt{34}12$ to
  $\patt{12}43$. Composing $S'$ and $S''$, we obtain a
  $\sigma$-stroll from $\patt{12}34$ to $\patt{12}43$ as required.
\end{proof}

\begin{proposition}\label{p:four}
  Let $n\geq 4$ and $W\sub V(H_n)$. If none of the conditions
  $(a)$--$(c)$ in Lemma~\ref{l:five} is satisfied, then $\repl{H_n}W$
  is $4$-colourable.
\end{proposition}
\begin{proof}
  By Lemma~\ref{l:max}, there is a set $Z$ such that $W\sub Z$ and
  $Z$ intersects each set $C_i$ in precisely one vertex. Since
  $\repl{H_n}Z$ contains $\repl{H_n}W$ as a subgraph, it is
  sufficient to prove the proposition under the
  assumption~\eqref{eq:total}.

  Let us therefore assume that \eqref{eq:total} holds for $W$, so the
  ensuing discussion applies. We retain its notation and
  definitions. By analyzing several cases, we will show that
  $\sigma^W$ is good, so the 4-colourability of $\repl{H_n}W$ follows
  from Lemma~\ref{l:good}. For the sake of a contradiction, suppose
  that $\sigma^W$ is not good.

  \begin{xcase}{$\sigma^W$ contains a nonzero even number of
      occurrences of the symbol 0.}%
    Considering the first two occurrences of 0 in $\sigma^W$, we find
    that $(\s0 \s0) \preceq \sigma^W$. Since $(\s0 \s0)$ is good
    (cf. Table~\ref{tab:good}) and $z(\s0 \s0) = 0 = z(\sigma^W)$,
    Lemma~\ref{l:sub}(i) implies that $\sigma^W$ is good, a
    contradiction.
  \end{xcase}

  \begin{table}
    \centering
    \begin{tabular}{l|l}
      \hspace{\fill}$\sigma$\hspace{\fill} & \hspace{\fill}$S$\hspace{\fill} \\\hline
      $\sseq{\s+ \s- \s+}$ & $\patt{12}34$, $\patt{14}23$, $\patt{24}31$, $\patt{12}43$ \\
      $\sseq{\s+ \s+ \s+ \s+}$ & $\patt{12}34$, $\patt{14}23$, $\patt{34}12$, $\patt{24}31$, $\patt{12}43$\\
      $\sseq{\s+ \s+ \s+ \s-}$ & $\patt{12}34$, $\patt{14}23$, $\patt{13}42$, $\patt{23}14$, $\patt{12}43$\\
      $\sseq{\s+ \s+ \s- \s-}$ & $\patt{12}34$, $\patt{14}23$, $\patt{34}12$, $\patt{13}24$, $\patt{12}43$\\
      $\sseq{\s+ \s- \s- \s-}$ & $\patt{12}34$, $\patt{14}23$, $\patt{24}31$,
      $\patt{23}14$, $\patt{12}43$\\
      $\sseq{\s0 \s0}$ & $\patt{12}34$, $\patt{34}12$, $\patt{12}43$\\ 
      $\sseq{\s0 \s+ \s+}$ & $\patt{12}34$, $\patt{34}12$, $\patt{24}31$,
      $\patt{12}43$\\
      $\sseq{\s0 \s+ \s0 \s0 \s-}$ & $\patt{12}34$, $\patt{34}12$, $\patt{23}41$, $\patt{14}32$, $\patt{23}14$, $\patt{12}43$
    \end{tabular}
    \caption{Some good sign sequences $\sigma$ and corresponding $\sigma$-strolls $S$.}
    \label{tab:good}
  \end{table}

  \begin{xcase}{$\sigma^W$ contains no occurrence of the symbol $0$.}%
    In view of Lemma~\ref{l:symm}, we may assume that $s_0 = +$. If
    $(\s+ \s- \s+)\preceq\sigma^W$, then $\sigma^W$ is good by
    Lemma~\ref{l:sub}(i) and the fact that $(\s+ \s- \s+)$ is good
    (see Table~\ref{tab:good}). Thus, $(\s+ \s- \s+) \not\preceq
    \sigma^W$. Since $n\geq 4$, we may consider the subsequence
    $\sigma' = (s_0,s_1,s_2,s_3)$ of $\sigma^W$ of length 4. To avoid
    an occurrence of the sequence $(\s+ \s- \s+)$, we necessarily have
    \begin{equation*}
      \sigma' \in \Setx{(\s+ \s+ \s+ \s+),(\s+ \s+ \s+ \s-),(\s+ \s+
        \s- \s-),(\s+ \s- \s- \s-)}.
    \end{equation*}
    Table~\ref{tab:good} shows that each possible value for $\sigma'$
    is a good sign sequence. Since $\sigma'\preceq\sigma$ and
    $z(\sigma') = 0 = z(\sigma^W)$, $\sigma^W$ is good by
    Lemma~\ref{l:sub}(i). This is a contradiction.
  \end{xcase}

  \begin{table}
    \centering
    \begin{tabular}{l|l}
      \hspace{\fill}$\sigma$\hspace{\fill} & \hspace{\fill}$S$\hspace{\fill} \\\hline
      $\sseq{\s0 \s+ \s0 \s+}$ & $\patt{12}34$, $\patt{34}12$, $\patt{23}41$,
      $\patt{14}23$, $\patt{34}12$\\
      $\sseq{\s0 \s+ \s0 \s-}$ & $\patt{12}34$, $\patt{34}21$, $\patt{13}42$,
      $\patt{24}31$, $\patt{34}12$\\
    \end{tabular}
    \caption{Some reversing sign sequences $\sigma$ and
      corresponding $\sigma$-strolls $S$. The $\sigma$-strolls to
      $\patt{34}21$ can be obtained by interchanging colours $1$ and
      $2$ in all the patterns.}
    \label{tab:reversing}
  \end{table}

  \begin{xcase}{$z(\sigma^W) = 1$.}%
    Applying a suitable symmetry of the graph $H_n$, and using the
    fact that $W$ does not satisfy conditions (b), (c) in
    Lemma~\ref{l:five}, we may assume that $s_0 = 0 \neq s_1$. In view
    of Lemma~\ref{l:symm}, it may further be assumed that $s_1 = +$.
    
    The sequence $\sseq{\s0 \s+ \s+}$ is good and we have
    $z(\sseqx{\s0 \s+ \s+}) = z(\sigma^W)$. Consequently, $\sseq{\s0
      \s+ \s+} \not\preceq \sigma^W$, and by symmetry, $(\s0 \s- \s-)
    \not\preceq \sigma^W$. In particular, none of $s_2,s_3$ is the
    symbol $+$ and at least one of $s_2,s_3$ is different from $-$. It
    follows that $0\in\Setx{s_2,s_3}$. Choose the least $j$ such that
    $j \geq 2$ and $s_j = 0$.

    We claim that there is $k > j$ such that $s_k \neq 0$. Suppose the
    contrary. Since the sum of all $s_i$ ($i\in\upto n$) is
    \begin{equation*}
      \sum_{i=0}^{n-1} s_i = (w_1-w_0)+(w_2-w_1)+\dots+(w_{n-1}-w_{n-2})+(-w_0-w_{n-1}) = w_0,
    \end{equation*}
    we find that there are two possibilities: either $\sigma^W =
    \sseq{\s0 \s+ \s- \s0 \s0\cdots \s0}$ and $w_0 = 0$, or $\sigma^W =
    \sseq{\s0 \s+ \s0 \s0\cdots \s0}$ and $w_0 = 1$. In the first case, however, $W$
    would satisfy condition (b) in Lemma~\ref{l:five}, while in the
    second case, condition (c) would be satisfied, a contradiction.

    Let us choose the least $k$ such that $k > j$ and $s_k \neq
    0$. Assume first that $s_k = +$. This implies that $k-j$ is odd,
    since otherwise $\sseq{\s0 \s+ \s+} \preceq \sigma^W$ and as we
    have seen, this would mean that $\sigma^W$ is good. However, if
    $k-j$ is odd, then $\sseq{\s0 \s+ \s0 \s+}\preceq\sigma^W$ and we
    get a contradiction with Lemma~\ref{l:sub}(ii) as $\sseq{\s0 \s+
      \s0 \s+}$ is reversing (cf. Table~\ref{tab:reversing}) and
    $z(\sseqx{\s0 \s+ \s0 \s+}) \neq z(\sigma^W)$.

    It remains to consider the possibility that $s_k = -$. If $k-j$ is
    odd, then for the reversing sequence $\sseq{\s0 \s+ \s0 \s-}$ we
    have $\sseq{\s0 \s+ \s0 \s-} \preceq \sigma^W$ and we obtain a
    contradiction with Lemma~\ref{l:sub}(ii) again. Thus, $k-j$ is
    even. In this case, we find $\sseq{\s0 \s+ \s0 \s0 \s-} \preceq
    \sigma^W$. As we can see from Table~\ref{tab:good}, $\sseq{\s0 \s+
      \s0 \s0 \s-}$ is good. Furthermore, $z(\s0 \s+ \s0 \s0 \s-) =
    z(\sigma^W)$, so $\sigma^W$ is good by Lemma~\ref{l:sub}(i), a
    contradiction.

    The discussion of Case 3, as well as the proof of
    Proposition~\ref{p:four}, is complete.
  \end{xcase}
\end{proof}

Theorem~\ref{t:main} is now an immediate consequence of
Lemma~\ref{l:five} and Proposition~\ref{p:four}.

We conclude this section by pointing out that the graph $H_4$ is the
only counterexample to Conjecture~\ref{conj:main} among edge-critical
graphs on up to 12 vertices, as was shown by a computer search using a
list of edge-critical graphs provided in~\cite{Roy}.


\section{Connection to monomial ideals}
\label{sec:algebra}

As mentioned in Section~\ref{sec:intro}, Conjecture~\ref{conj:main}
was motivated by questions arising from commutative algebra. It turns
out that the graph $H_4$ serves as a counterexample for two other
problems on the properties of square-free monomial ideals which we
state in this section. For the terms not defined here, as well as for
more information on commutative algebra and its relation to
combinatorics, see~\cite{MS-combinatorial}. Monomial ideals are the
subject of the monograph~\cite{HH-monomial}.

Let $R$ be a commutative Noetherian ring and $I \subseteq R$ an
ideal. A prime ideal $P$ is \emph{associated} to $I$ if there exists
an element $m \in R$ such that $P = I:\langle m \rangle$ (the ideal
quotient of $I$ and $\langle m \rangle$). The \emph{set of associated
  prime ideals} (associated primes) is denoted by
$\Ass(I)$. Brodmann~\cite{Bro-asymptotic} showed that $\Ass(I^s) =
\Ass(I^{s+1})$ for all sufficiently large $s$. The ideal $I$ is said
to have the \emph{persistence property} if
\begin{equation*}
  \Ass(I^s) \subseteq \Ass(I^{s+1})  
\end{equation*}
for all $s \geq 1$.

Let $k$ be a fixed field and $R = k[x_1,\dots,x_n]$ a polynomial ring
over $k$. An ideal in $R$ is \emph{monomial} if it is generated by a
set of monomials. A monomial ideal is \emph{square-free} if it has a
generating set of monomials where the exponent of each variable is at
most $1$. The question that motivated Francisco et
al.~\cite{FHT-conjecture} to pose Conjecture~\ref{conj:main} is the
following one (see~\cite[Question~3.28]{VT-beginners},
\cite[Question~4.16]{MV-edge} or \cite{HRV-stable,MMV-associated}):
\begin{problem}\label{prob:persistence}
  Do all square-free monomial ideals have the persistence property?
\end{problem}

Francisco et al.~\cite{FHT-conjecture} proved that if
Conjecture~\ref{conj:main} holds, then the answer to
Problem~\ref{prob:persistence} is affirmative. While our
counterexample to Conjecture~\ref{conj:main} does not necessarily
imply a negative answer to Problem~\ref{prob:persistence}, the cover
ideal of $H_4$ does in fact show that the answer is negative.

Given a graph $G$, a \emph{transversal} (or \emph{vertex cover}) of
$G$ is a subset $T \subseteq V(G)$ such that every edge of $G$ has an
end vertex in $T$. If $V(G)=\{x_1, \ldots, x_n\}$, we can associate
each $x_i$ with the variables in the polynomial ring $k[x_1, \ldots,
x_n]$. The \emph{cover ideal} $J(G)$ is the ideal generated by all
inclusion-wise minimal transversals of $G$.

Let $J=J(H_4)$ denote this cover ideal in the polynomial ring $R =
k[x_1,\ldots,x_{12}]$, where $H_4$ is the graph on 12 vertices defined
in Section~\ref{sec:counterexample}. Using the commutative algebra
program Macaulay2~\cite{GS-macaulay2}, we can compute the set of
associated primes of $J^3$ and $J^4$.  By comparing the output, one
finds that
\begin{equation*}
  \Ass(J^4) = \Ass(J^3) - \Setx{M},
\end{equation*}
where $M$ is the maximal ideal of $R$. In particular:
\begin{theorem}\label{t:persistence}
  The cover ideal $J(H_4)$ does not have the persistence property.
\end{theorem}

The second question concerns the depth function of monomial ideals. If
$I$ is an ideal in $R$, then the \emph{depth function} of $I$ is the
function $f:\mathbb N \rightarrow \mathbb N$ defined by 
\begin{equation*}
  f(s) = \depth(R/I^s),
\end{equation*}
where $\depth(\cdot)$ is the depth of a ring as defined,
e.g., in~\cite[Chapter~6]{Mat-commutative}.

Herzog and Hibi~\cite{HH-depth} noted that the depth function of most
monomial ideals is non-increasing, but they constructed examples where
this is not the case (for instance, one where the depth function is
non-monotone). They asked the following question:
\begin{problem}\label{prob:depth}
  Do all \emph{square-free} monomial ideals have a non-increasing depth
  function?
\end{problem}
(See also~\cite{BHH-monomial,HRV-stable}.) As noted
in~\cite{BHH-monomial}, the question of Problem~\ref{prob:depth} is a
natural one since a monomial ideal $I$ satisfies the persistence
property if all monomial localisations of $I$ have a non-increasing
depth function. According to~\cite{BHH-monomial}, a positive answer
was `expected'.

However, the cover ideal of $H_4$ again provides a
counterexample. Using Macaulay2 we find that
\begin{equation*}
  \depth (R/J^3) = 0 < 4 = \depth (R/J^4),
\end{equation*}
so we have the following:
\begin{theorem}\label{t:depth}
  The depth function of the cover ideal $J(H_4)$ is not
  non-increasing.
\end{theorem}


\section{Acknowledgements}

We would like to thank Chris Francisco, T\`{a}i H\`{a}, and Adam Van
Tuyl for their many helpful comments and suggestions, and particularly
for pointing out that the cover ideal of $H_4$ provides negative
answers to Problems~\ref{prob:persistence} and~\ref{prob:depth}.


\end{document}